\documentclass[11pt]{article}
\usepackage[top=1in, bottom=1in, left=1in, right=1in]{geometry}
\usepackage{graphicx}
\usepackage{multicol}
\usepackage{times}
\usepackage{setspace} 
\usepackage{here}
\usepackage{amsfonts, amssymb}
\usepackage{mathrsfs}
\usepackage{amsthm}
\usepackage{amsmath}
\usepackage{bbm}
\numberwithin{equation}{section}
\newtheorem{theorem}{Theorem}[section]
\newtheorem{proposition}{Proposition}[section]
\newtheorem{lemma}{Lemma}[section]
\theoremstyle{definition}
\newtheorem{dy}{Definition}[section]
\newtheorem{remark}{Remark}[section]

\usepackage{amsfonts, mathtools}
\usepackage{amsmath,amssymb}
\usepackage{hyperref}
\hypersetup{%
	colorlinks=true,
	linkcolor=black,
	citecolor=black, 
	urlcolor=blue
}
\usepackage{bxtexlogo}

\usepackage{listings}
\usepackage[edges]{forest}
\newdimen\forestdimen

\title{Homeomorphism of the Revuz correspondence under Dynkin class assumptions}
\author{Zijian Xu\thanks{Zijian Xu (xuzijian823@gmail.com). Department of Applied Mathematics, Fukuoka University,
		Fukuoka 814-0180, Japan. Supported in part by fund  from the Central Research Institute of Fukuoka University.}}
\date{}
\begin{document}
	\maketitle
		\begin{abstract}
		This paper investigates the topological properties of the Revuz correspondence between positive continuous additive functionals (PCAFs) and their associated smooth measures. Within the Dynkin, local Dynkin, and Green-tight Dynkin classes, we establish bidirectional equivalences among measure convergence, potential convergence, and PCAF convergence. In the local Dynkin class, weak convergence on compact sets, strong $\mathcal{E}_1$-convergence of potentials, uniform convergence of potentials, and $L^1$-convergence of PCAFs are mutually equivalent; under the Green-tight condition, this equivalence extends to the whole space.
	\end{abstract}
	\section{Introduction}
	In the present article, we focus on the convergence of positive continuous additive functionals (PCAFs) in terms of their associated smooth measures. The one-to-one correspondence between PCAFs and smooth measures, known as the Revuz correspondence, has long been a cornerstone of Dirichlet form theory, yet its topological aspects-namely, how convergence on one side relates to convergence on the other—have remained largely unexplored. This question has recently gained renewed interest due to applications in the construction of Liouville-type stochastic processes, where the limiting smooth measures are typically singular and the corresponding PCAFs admit no explicit representation.\\
	Let $E$ be a locally compact separable metric space and $\mathfrak{m}$ a positive Radon measure on $E$ with full support. We consider an $\mathfrak{m}$-symmetric Hunt process $\mathbb{M}=(X_t,\mathbb{P}_x)$ on $E$ whose associated Dirichlet form $(\mathcal{E},\mathcal{F})$ on $L^2(E;\mathfrak{m})$ is assumed to be regular. The Revuz correspondence then provides a bijection between the class of PCAFs $\mathbf{A}_c^+$ and the class of smooth measures ${\sf S}$. For instance, for a nonnegative bounded Borel function $f$, the PCAF $\displaystyle{\sf A}_t=\int_0^t f(X_s){\rm d}s$ corresponds to the smooth measure $f(x)\mathfrak{m}({\rm d}x)$; for one-dimensional Brownian motion, the Dirac measure $\delta_a$ corresponds to the local time at $a$.\\
	Despite the algebraic clarity of this correspondence, the vastness of both ${\sf S}$ and $\mathbf{A}_c^+$ poses a significant obstacle to establishing a unified topological framework. The class ${\sf S}$ contains not only all Radon measures charging no zero-capacity sets, but also many non-Radon Borel measures; the class $\mathbf{A}_c^+$ is equally broad. A natural strategy is to first investigate well-behaved subclasses and then gradually extend the scope.\\
	Nishimori, Tomisaki, Tsuchida, and Uemura \cite{Nishimori_Tomisaki_Tsuchida_Uemura_2025} initiated this program by focusing on the subclass ${\sf S}_0$ consisting of smooth measures of finite energy integrals. For $\mu\in{\sf S}_0$, the 1-potential $U_1\mu$ belongs to the Dirichlet space $\mathcal{F}$, and they introduced a metric $\rho$ on ${\sf S}_0$ via the $\mathcal{E}_1$-norm of the difference of 1-potentials. They showed that $({\sf S}_0,\rho)$ is a Polish space and proved that if $\rho(\mu_n,\mu)\to 0$, then there exists a subsequence of the corresponding PCAFs that converges $\mathbb{P}_x$-almost surely in the local uniform topology for quasi-every starting point. This provided a positive-direction compactness result, but did not address the converse implication.\\
	Subsequently, Ooi \cite{MR4968118} substantially strengthened this result by establishing the converse direction within ${\sf S}_0$. Through the Fukushima decomposition and the Beurling-Deny decomposition, he derived a key energy identity involving the killing measure $\kappa$ and an energy functional $\nu_0$ describing the part of the process that continuously escapes to the cemetery. Using this identity, he proved that the Revuz map from $({\sf S}_0,\rho)$ to the corresponding PCAF class equipped with the $L^2(\mathbb{P}_{m+\kappa+\nu_0})$-topology with local uniform convergence is in fact a homeomorphism. This remains the first and only complete topological characterization of the Revuz correspondence, but it is confined to ${\sf S}_0$.\\
	Independently, Noda \cite{MR5063570} pursued a different approach for processes satisfying the absolute continuity condition—that is, processes admitting transition densities with respect to $\mathfrak{m}$. His framework restricts the measures to the smaller class ${\sf S}_{00}$, consisting of finite smooth measures with bounded 1-potential. He derived a quantitative estimate bounding the expected supremum distance between two PCAFs by the sup-norm distance between their 1-potentials, and obtained positive-direction convergence results: if the 1-potentials converge uniformly (or locally uniformly, under conservativeness), then the associated PCAFs converge in $L^2(\mathbb{P}_x)$ expectation or in probability, respectively. His results hold for every starting point, but are strictly one-directional and rely crucially on the existence of transition densities.\\
	To extend the analysis beyond ${\sf S}_0$ and ${\sf S}_{00}$ to the full smooth measure class ${\sf S}$, Ooi, Uemura, and Tsuchida \cite{ooi2025smooth} introduced the notion of smooth measures attached to a compact nest $\{F_\ell\}$. Since every smooth measure admits a nest such that its restrictions to $F_\ell$ belong to ${\sf S}_0$, they defined a metric on the projective limit of these restriction classes and studied several modes of convergence—weak convergence on the nest, vague convergence on the nest, and resolvent-sense convergence. They obtained sufficient conditions under which convergence of PCAFs implies convergence of the corresponding measures. This work represents the first systematic attempt to handle general smooth measures, but the results are conditional and require additional assumptions on the tails of the measures with respect to the nest.\\
	The present paper investigates the topological properties of the Revuz correspondence under Dynkin-class assumptions. The Dynkin class ${\sf S}_D$ consists of smooth measures whose $\alpha$-potential is globally bounded; the local Dynkin class ${\sf S}_{L\!D}$ only requires boundedness on each compact set and the Green-tight Dynkin class ${\sf S}_{D_\infty}^\alpha$ further requires uniformly small tail potential outside a sufficiently large compact set. These classes arise naturally in the study of Schr\"odinger operators and Feynman–Kac semigroups, and they are not contained in ${\sf S}_0$ in general.\\
	Our main results establish two-directional equivalences within these classes. Theorem 4.1 shows that if $\mu\in{\sf S}_D$, $\{\mu_n\}$ is uniformly Dynkin, and the corresponding PCAFs converge almost surely locally uniformly, then $R_1\mu_n$ converges pointwise to $R_1\mu$ quasi-everywhere, a converse-direction result without requiring finite energy integrals or transition densities. Theorem 4.2 establishes, under the continuity of the resolvent kernel $R_1(\cdot,\cdot)$ on $E\times E$ and within the local Dynkin class, a full equivalence among weak convergence of measures on compact sets, strong $\mathcal{E}_1$-convergence of potentials, uniform convergence of potentials, and $L^1(\mathbb{P}_x)$ convergence of PCAFs for every starting point. Theorem 4.3 lifts this equivalence to the whole space under the Green-tight condition, yielding a global homeomorphism of the Revuz map on ${\sf S}_{G\!D}^1$. Collectively, these results extend Ooi's homeomorphism from ${\sf S}_0$ to the larger local Dynkin class, supplement Noda's one-directional results by providing the converse implication without assuming transition densities, and upgrade the subsequential convergence of Nishimori et al. to full sequential convergence under appropriate assumptions.\\
	The structure of this article is as follows. Section 2 collects the necessary preliminaries on Dirichlet forms, Hunt processes, smooth measures, PCAFs, and the various Dynkin classes. Section 3 establishes technical lemmas, including a martingale additive functional decomposition for potentials of measures in ${\sf S}_{00}$, equicontinuity properties of resolvent kernels under the continuity assumption, and a key proposition showing that local Dynkin measures restricted to compact sets belong to ${\sf S}_{00}$. Section 4 presents and proves the three main theorems.
	\section{Framework}
	Let $(E;{\sf d})$ be a locally compact separable metric space and  $\mathfrak{m}$ a positive Radon measure on $E$ with full topological support. Let $E_\partial:= E \cup \{\partial\}$ be the one-point compactification of $E$. When $E$ is already compact, $\partial$ is regarded as an isolated point. Let $\mathbb{M}=(\Omega,\mathscr{F},\mathscr{F}_t,X_t,\mathbb{P}_x,\zeta)$ be an $\mathfrak{m}$-symmetric Hunt process on $E$  where $\{\mathscr{F}_t\}$ is the minimum augmented filtration and $\zeta$ is the life time associated with $\mathbb{M}$, i.e. $\zeta(\omega):=\inf\{ t\geq 0 \mid X_t = \partial  \}$.
	\par Let $\{P_t\}_{t\geq0}$ be the transition semigroup of $\mathbb{M}$, i.e. $P_t f(x)=\mathbb{E}_x[f(X_t)]. $ There exists a strongly continuous Markovian semigroup $\{T_t\}_{t\geq 0}$ on $L^2(E;\mathfrak{m})$ associated $\{P_t\}_{t\geq 0}$. We define the Dirichlet form on $L^2(E;\mathfrak{m})$ of $\mathbb{M}$ as
	\[
	\left\{
	\begin{aligned}
	&\mathcal{E}(u,v)=\lim_{t \downarrow 0}\frac{1}{t}(u-T_t u,v),\\
	&\mathcal{F}=\left\{u \in L^2(E;\mathfrak{m}) \ \left|\ \lim_{t \downarrow 0} \frac{1}{t}(u-T_t u,u)<\infty \right.\right\},
	\end{aligned}
	\right.
	\]
	where $(f,g)$ means the $L^2$-inner product of  $f$ and $g$. For $\alpha>0$ set $\mathcal{E}_\alpha(u,v):=\mathcal{E}(u,v)+\alpha(u,v)$ for $u,v\in\mathcal{F}$. Note that $\mathcal{F}$ is Hilbert space with inner product $\mathcal{E}_1$ and $\Vert\cdot \Vert_{\mathcal{E}_1}$ is the norm of $\mathcal{F}$ with respect to $\mathcal{E}_1$, i.e. $\Vert f \Vert_{\mathcal{E}_1}:=\sqrt{\mathcal{E}_1 (f,f)}$ for $f\in\mathcal{F}$. We said that Dirichlet form $(\mathcal{E},\mathcal{F})$ is \textit{regular} if $\mathcal{F} \cap C_0(E)$ is dense in $\mathcal{F}$ with respect to $\mathcal{E}_1$-norm and is also dense in $C_0(E)$ with respect to $\Vert \cdot \Vert_\infty$, where $C_0(E):=\{ u\in C(E)\mid {\rm supp}[u] \text{ is compact} \}$,  $C(E)$ is the space of all real continuous functions on $E$ and $\displaystyle\Vert f \Vert_\infty:=\sup_{x\in E}|f|$. Throughout this paper, we assume that $(\mathcal{E},\mathcal{F})$ is regular Dirichlet form. The transition kernel of $\mathbb{M}$ is denoted by $P_t(x,{\rm d}y)$, $t>0$. The correspodence between $\mathbb{M}$ and $(\mathcal{E},\mathcal{F})$ is given by
	\[
	T_t f(x)=\mathbb{E}_x\left[  f(X_t): t<\zeta  \right]:=\int_E f(y) P_t(x,{\rm d}y)\quad\mathfrak{m}\text{-a.e. }x\in E
	\]
	Here and in the sequel, unless mentioned otherwise, we use the convention that a function defined on $E$ takes the value $0$ at $\partial$.  The process $\mathbb{M}$ is said to have the \textit{absolute continuity condition} with respect to $\mathfrak{m}$ (\textbf{(AC)} in abbreviation) if for any $x \in E$ and $t>0$, $\mathfrak{m}(A)$=0 implies $P_t(x,A)=0$ for all $A \in \mathscr{B}(E)$. Throughout this paper, $\mathbb{M}$ is assumed to have $\mathbf{(AC)}$.
	\par For an open set $O \subset E$, we define the capacity associated with the Dirichlet form $(\mathcal{E},\mathcal{F})$ as
	\[
	{\rm Cap}(O):=\left\{
	\begin{aligned}
		&\inf\{ \mathcal{E}_1(u,u) \mid  u \in \mathcal{L}_O  \},&\quad&\text{if}&\quad&\mathcal{L}_O \neq \emptyset,&\\
		&\infty& &\text{if}& &\mathcal{L}_O = \emptyset,&
	\end{aligned}
	\right.
	\]
	where $\mathcal{L}_O:=\{ u \in \mathcal{F} \mid u \geq 1 \ \mathfrak{m}\text{-a.e. on }O\}$, and for ant set $A \subset E$,
	\[
	{\rm Cap}(A):=\inf\left\{ {\rm Cap}(O) \mid A \subset O,\ O\text{ is open}   \right\}.
	\]
	Then it is known that Cap is a Choquet capacity (see \cite{MR2778606}). A statement depending on $x \in A$ is said to hold q.e. on $A$ if there exists a set $N \subset A$ of zero capacity such that the statement is true for every $x \in A \setminus N$. ``q.e." is an abbreviation of \textit{quasi-everywhere}. Let $u$ be an extended real valued function defined q.e. on $E$. We call $u$ \textit{quasi continuous} if there exists for any $\varepsilon > 0$ an open set $G \subset E$ such that ${\rm Cap}(G) < \varepsilon$ and $u |_{E \setminus G}$ is finite continuous. Here $u |_{E \setminus G}$ denotes the restriction of $u$ on $E \setminus G$. A sequence $\{F_k\}$ of closed sets such that $F_k \uparrow$ and ${\rm Cap}(E\setminus F_k) \downarrow 0$ as $k \to \infty $ is called a \textit{nest} on $E$.
	\par Let us call a (positive) Borel measure $\mu$ on $E$ \textit{smooth} if it satisfies the following conditions:
	\begin{itemize}
		\item[$(S.1)$] $\mu$ charges no set of zero capacity.
		\item[$(S.2)$] There exists an increasing sequence $\{F_n\}$ of closed sets such that
		\begin{equation}\label{2.1}
			\mu(F_n)<\infty,\quad n=1,2,\dots,
		\end{equation}
		\begin{equation}\label{2.2}
			\lim_{n\to\infty}{\rm Cap}(K\setminus F_n)=0 \quad \text{for any compact set }K.
		\end{equation}
	\end{itemize}
	We denote the set of all smooth measures by ${\sf S}$. Let us note that  $\mu$ then satisfies
	\begin{equation}\label{2.3}
		\mu \left( E \setminus \bigcup_{n=1}^\infty F_n  \right)=0
	\end{equation}
	An increasing sequence $\{F_k\}$ of closed sets satisfying condition (\ref{2.2}) will be called a \textit{generalized nest} to distinguish it from the nest introduced in the preceding section. If further each $F_n$ is compact we call it a \textit{generalized compact nest}. A positive Radon measure $\mu$ on $E$ is said to be of \textit{finite energy integrals} if there exists a positive constant $C>0$ such that
	\[
	\int_E \left| v(x) \right| \mu({\rm d}x) \leq C\sqrt{\mathcal{E}_1(v,v)},\quad\text{for all}\quad v\in \mathcal{F}\cap C_0(E).
	\]
	The set of all measures on $E$ which are of finite energy integrals is denoted by ${\sf S}_0$. Then, according to the Riesz representation theorem, we find that $\mu \in {\sf S}_0$ if and only if there exists a unique element $U_\alpha \mu \in \mathcal{F}$, called the $\alpha$-\textit{potential} of $\mu$, for each $\alpha>0$ such that
	\[
	\mathcal{E}_\alpha(U_\alpha \mu,\nu)=\int_E v(x) \mu({\rm d}x), \quad\text{for all}\quad v \in \mathcal{F}\cap C_0(E).
	\]
	A subclass ${\sf S}_{00}$ of ${\sf S}_0$ is defined as
	\[
	{\sf S}_{00}:=\{\mu \in {\sf S}_0 \mid \mu(E)<\infty, \Vert U_1 \mu \Vert_\infty<\infty\}.
	\]
	For any Borel set $A \subset E$, set $1_{A} \mu=:\mu^A$ for $\mu \in {\sf S}$.  Define a function $\rho:{\sf S}_0 \times {\sf S}_0\to[0,+\infty)$ by
	\[
	\rho(\mu,\nu):=\Vert U_1 \mu - U_1\nu \Vert_{\mathcal{E}_1}\quad \text{for}\quad \mu,\nu\in{\sf S}_0.
	\]
	That the  function $\rho$ becomes a metric on ${\sf S}_0$ and the metric space $(\rho,{\sf S}_0)$ is a Polish space (see \cite[Proposition 3.6]{Nishimori_Tomisaki_Tsuchida_Uemura_2025}). If there exists uniquely a positive Radon measure $\mu_{\langle u \rangle}, u\in\mathcal{F}_b$, satisfying
	\[
	\int_E f(x)\ {\rm d}\mu_{\langle u \rangle}=2\mathcal{E}(uf,u)-\mathcal{E}(u^2,f)\quad\text{for}\quad f\in\mathcal{F}\cap C_0(E).
	\]
	$\mu_{\langle u \rangle}$ is finite by the above estimate. We call $\mu_{\langle u \rangle}$ the \textit{energy measure} of $u \in \mathcal{F}_b$. We know that $\mu_{\langle u \rangle}(E)=2\mathcal{E}(u,u)$ for $u \in \mathcal{F}_b$.
	\begin{theorem}[\cite{MR2778606} Theorem 2.2.3]\label{thm2.1}
		The following conditions are equivalent for a Borel set $B \subset E$:
		\begin{itemize}
			\item [{\rm (i)}] ${\rm Cap}(B)=0$.
			\item[{\rm (ii)}] $\mu(B)=0$, $\forall \mu \in {\sf S}_0$.
			\item[{\rm (iii)}] $\mu(B)=0$, $\forall \mu \in {\sf S}_{00}$.
		\end{itemize}
	\end{theorem}
	\begin{theorem}[\cite{MR2778606} Theorem 2.2.4]\label{thm2.2}
		The following conditions are equivalent for a positive Borel measure $\mu$ on $E$.\\
		$({\rm i})$ $\mu$ $\in$ ${\sf S}$.\\
		$({\rm ii})$ There exists a generalized nest $\{F_k\}$ satisfying (\ref{2.3}) and $\mu^{F_k} \in {\sf S}_0$ for each $k$.\\
		$({\rm iii})$ There exists a generalized compact nest $\{F_k\}$ satisfying (\ref{2.3}) and $\mu^{F_k} \in {\sf S}_{00}$ for each $k$.
	\end{theorem}

		For any $\mu_n,\mu \in {\sf S}_0$ and for all compact set $K \subset E$, we say $\mu_n$ converges to $\mu$ \text{weakly} on $K \subset E$ if for any $\varphi \in C(K)$,
		\[
		\lim_{n\to\infty}\int_K \varphi \ {\rm d}\mu_n=\int_K \varphi \ {\rm d}\mu.
		\]
		For any $\mu_n,\mu \in {\sf S}_0$, we say $\mu_n$ converges to $\mu$ \text{vaguely} if for any $\varphi \in C_0(E)$,
		\[
		\lim_{n\to\infty}\int_E \varphi \ {\rm d}\mu_n=\int_E \varphi \ {\rm d}\mu.
		\]
		
	\begin{proposition}[\cite{Nishimori_Tomisaki_Tsuchida_Uemura_2025} Proposition3.8]\label{prop 2.1}
		Assume that $\{ \mu_n\}$ are measures in ${\sf S}_0$ which converges to some $\mu \in {\sf S}_0$ in $\rho$. Then $\mu_n$ converges to $\mu$ vaguely.
	\end{proposition}
	\begin{proposition}[\cite{Nishimori_Tomisaki_Tsuchida_Uemura_2025} Proposition3.9]
		Let $\{\mu_n \}$ be a sequence in ${\sf S}_0$. Assume that $\{ \mu_n\}$ is bounded with respect to $\rho$, that is, $\sup_n\mathcal{E}(U_1 \mu_n, U_1 \mu_n)<\infty$. If that measures $\{\mu_n\}$ converges to a measure $\mu$ vaguely, then $\mu \in {\sf S}_0$ and $\{ \mu_n\}$ converges to $\mu$ weakly with respact to $\rho$.
	\end{proposition}
	\begin{proposition}[\cite{Nishimori_Tomisaki_Tsuchida_Uemura_2025} Proposition3.10]
		 \begin{itemize}
			\item[{\rm (i)}] (monotonicity) Let $\mu$ and $\nu$ be Borel measures on $(E,\mathscr{B}(E))$. Assume that $\mu \leq \nu$, that is, $\mu(A) \leq \nu(A)$ for $A \in \mathscr{B}(E)$. If $\nu \in {\sf S}_0$, then $\mu \in {\sf S}_0$ and $U_\alpha \mu \leq U_\alpha \nu$.
			\item[{\rm (ii)}] (convex cone) If $\mu,\nu \in {\sf S}_0$, $a,b\geq 0$ and $a > 0$, then  $a \mu + b \nu \in {\sf S}_0$ and $U_\alpha (a \mu + b \nu)= a U_\alpha \mu + b U_\alpha \nu$.
			\item[{\rm (iii)}] (ideal) Let $\mathscr{B}_{b,+}(E)$ be the set of all nonnegatiove bounded measurable functions on $E$. Then $\mathscr{B}_{b,+}(E)$ is an ideal in $\mathcal{S}_0$ (i.e. $\mathscr{B}_{b,+}(E){\sf S}_0 \subset {\sf S}_0$) in the sense that $f\mu$ belongs to ${\sf S}_0$ whenever $\mu \in {\sf S}_0$ and $f \in \mathscr{B}_{n,+}(E)$.
		\end{itemize}
	\end{proposition}
	Under assumption \textbf{(AC)}, there exists an $\alpha$-order resolvent kernel $R_\alpha (x,y)$ which is defined for all $x,y \in E$ (see \cite[Lemma 4.2.4]{MR2778606}). For a non-negative Borel measure $\mu$, we write $R_\alpha \mu (x):= \int_E R_\alpha (x,y)\mu({\rm d}y)$. In this case, the function $R_\alpha \mu(x)$ is a quasi continuous and $\alpha$-excessive version of the $\alpha$-potential $U_\alpha \mu$ of $\mu \in {\sf S}_0$, $\alpha>0$.
	\begin{dy}[Dynkin class]
		A positive smooth measure $\mu\in{\sf S}$ is said to be of \textit{Dynkin class} ${\sf S}_{D}$ if 
		\[
		\Vert R_\alpha \mu \Vert_\infty <\infty, \quad \text{for some / any}\quad \alpha>0.
		\]
	\end{dy}
	\begin{dy}[Uniformly Dynkin]
		A family  $\{\mu_n\}$ of positive smooth measures is said to be of \textit{uniformly  local Dynkin} if 
		\[
		\sup_{n\in \mathbb{N}} \Vert R_\alpha \mu_n \Vert_\infty <\infty, \quad \text{for some / any}\quad \alpha>0.
		\]
	\end{dy}
		\begin{dy}[Local Dynkin class]
		A positive smooth measure $\mu\in{\sf S}$ is said to be of \textit{local Dynkin class} ${\sf S}_{L\!D}$ if $1_K \mu \in {\sf S}_D$ for any compact set $K \subset E$, i.e.,
		\[
		\Vert R_\alpha (1_K\mu) \Vert_\infty <\infty, \quad \text{for some / any}\quad \alpha>0.
		\]
	\end{dy}
	\begin{dy}[Uniformly local Dynkin]
		A family $\{\mu_n\}$ of positive smooth measures is said to be of \textit{uniformly  local Dynkin} if $1_K \mu_n$ is of uniformly Dynkin for any  compact set $K \subset E$, i.e.,
		\[
		\sup_{n\in \mathbb{N}} \Vert R_\alpha (1_K\mu_n) \Vert_\infty <\infty, \quad \text{for some / any}\quad \alpha>0.
		\]
	\end{dy}
	\begin{dy}[Green-tight  Dynkin class measure]
		$\mu$ is said to be an \textit{$\alpha$-order Green-tight measure of Dynkin class with respect to} $\mathbb{M}$ if $\mu \in {\sf S}_D$ and for any $\varepsilon>0$ and $\alpha>0$  there exists a compact subse $K=K(\varepsilon)$ of $E$ such that 
		\[
		\sup_{x \in E} R_\alpha (1_{K^c} \mu)(x)<\varepsilon.
		\]
	\end{dy}
	The set of all $\alpha$-order Green-tight measure of Dynkin class with respect to $\mathbb{M}$ on $E$ is denoted by  ${\sf S}_{D_\infty}^\alpha$. For more details on the definitions and properties of Green-tight, please refer to \cite{MR1926893} and \cite{MR3747482}.
	\begin{dy}[Uniformly Green-tight Dynkin measures]
		A family of Dykin measures $\{\mu_n\}$ is said to be  of  \textit{$\alpha$-order uniformly Green-tight measures of   Dynkin with respect to} $\mathbb{M}$ if $\{\mu_n\}$ is of uniformly Dynkin and for any $\varepsilon>0$ and $\alpha>0$  there exists a compact subse $K=K(\varepsilon)$ of $E$ such that 
		\[
		\sup_{n \in \mathbb{N}}\sup_{x \in E} R_\alpha (1_{K^c} \mu_n)(x)<\varepsilon\quad \text{and}\quad \sup_{n\in \mathbb{N}} \int_{K^c} R_\alpha  (1_{K^c} \mu_n) \ {\rm d}\mu_n<\varepsilon.
		\]
	\end{dy}
	The following inequality is called \textit{Stollmann-Voigt’s inequality} (see \cite[Theorem 3.1]{MR1378151}): for $ \mu \in {\sf S}_{D}$, $\alpha >0 $
	\begin{equation}\label{2.4}
		\int_E f^2\ {\rm d}\mu \leq \Vert R_\alpha {\mu} \Vert_\infty \mathcal{E}_\alpha(f,f),\quad\text{for}\quad f\in\mathcal{F}.
	\end{equation}
	Form (\ref{2.4}), any $\mu \in {\sf S}_D$ is a Radon measure on $E$, because of the regularity of the Dirichlet form (see \cite{MR3747482}). For $\mu \in {\sf S}_D$, $\Vert R_\alpha \mu \Vert_\infty < \infty$ for $\alpha>0$ (see \cite[Lemma 2.9]{MR5063570}). Similarly, if $\{\mu_n\}$ is of uniformly Dynkin, then $\sup_n \Vert R_\alpha \mu_n \Vert_\infty<\infty$ for $\alpha>0$. For any compact set $K \subset E$, since Dirichlet form $(\mathcal{E},\mathcal{F})$ is regular, then there exists $\varphi\in \mathcal{F} \cap C_0(E)$ such that $\varphi\equiv 1$ on $K$ and  $0\leq\varphi \leq1$. Then, from (\ref{2.4}), for any $\alpha>0$
	\[
	\mu_n(K) = \int_K \varphi^2 \ {\rm d}\mu_n \leq \Vert R_\alpha \mu_n \Vert_\infty \mathcal{E}_\alpha(\varphi,\varphi).
	\]
	Therefore, $\sup_n \mu_n(K)<\infty$ for any compact set $K\subset E$.
	\begin{proposition}
		If $\mu \in {\sf S}_{L\!D}$, then $\mu^K \in {\sf S}_{00}$ for any compact set $K\subset E$.
	\end{proposition}
	\begin{proof}
		For $\mu \in {\sf S}_{L\!D}$ and $K \subset E$ is a compact set, there exists  a generalized compact nest $\{F_k\}$ satisfying (\ref{2.3}) and $\mu^{F_k} \in {\sf S}_{00}$ for each $k$ by Theorem~\ref{thm2.2}. Set $G_k:= K\cap F_k$, then
		\[
		{\rm Cap}\left(K\setminus \bigcup_{k=1}^\infty G_k\right)\leq {\rm Cap}(K \setminus G_k) \to 0\quad \text{as}\quad k \to \infty.
		\]
		Since $\mu$ charges no sets of zero capacity, we have $\mu(K\setminus \bigcup_{k=1}^\infty G_k)=0$ which implies that
		\[
		\lim_{k\to\infty} \mu(K \setminus G_k)=0.
		\]
		For any $v \in \mathcal{F} \cap C_0(E)$,
		\[
		\int_E |v| \ {\rm d}\mu^K=\int_{K \setminus G_k} |v| \ {\rm d}\mu+\int_{ K \cap G_k} |v| \ {\rm d}\mu=:{\rm (I)}+{\rm (II)}.
		\]
		Since,
		\[
		{\rm (I)} \leq \Vert v \Vert_\infty  \cdot\mu(K\setminus G_k) \to 0 \quad \text{as}\quad k \to \infty,
		\]
		and
		\[
		\begin{aligned}
			{\rm (II)}&\leq \sqrt{\int_{K}  R_1 \mu^{K} \ {\rm d}\mu } \cdot \sqrt{\mathcal{E}_1 (v,v)}\\
			&\leq\sqrt{\Vert R_1 \mu^{K}\Vert_\infty \cdot \mu(K)} \cdot \sqrt{\mathcal{E}_1 (v,v)}.
		\end{aligned}
		\]
		Therefore, for any $v \in \mathcal{F} \cap C_0(E)$,
		\[
		\int_E |v| \ {\rm d}\mu^K \leq C\sqrt{\mathcal{E}_1(v,v)},
		\]
		where $C:= \sqrt{\Vert R_1 \mu^{K} \Vert_\infty\cdot \mu(K)} $.  Then $\mu^K \in {\sf S}_0$ and hence $\mu^K \in {\sf S}_{00}$.
	\end{proof}
	\begin{proposition}\label{prop2.5}
		Assume $\{\mu_n\} $ is of uniformly local Dynkin. If $R_1 (\cdot,\cdot)$ is continuous on $K\times K$ where $K$ is compact set of $E$. Then $R_1 \mu_n^K(x)$  is uniformly equicontinuous on $K$, that is , for any $\varepsilon>0$ there exists $\delta>0$ such that for any $n \in \mathbb{N}$ and $x,y\in K$ satisfying  ${\sf d}(x,y)<\delta$, we have $\left|  R_1 \mu_n^K(x) - R_1 \mu_n^K(y) \right|<\varepsilon$.
	\end{proposition}
	\begin{proof}
		Since $R_1(x,y)$ is continuous on $K \times K$, then $R_1(x,y)$ is uniformly continuous on $K \times K$. Then, for any $\tau>0$, there exists $\delta>0$ such that for any $x,y \in K$ satisfying ${\sf d}(x,y)<\delta$, we have
		\[
		\left|   R_1(x,z) - R_1 (y,z)  \right|<\tau\quad \text{ for any }\quad z \in K.
		\]
		Therefore, for any $x,y \in K$ satisfying ${\sf d}(x,y)<\delta$,
		\[
		\begin{aligned}
			\left| R_1 \mu_n^K(x)- R_1 \mu^K_n(y)   \right|&=\left| \int_K R_1(x,z) \mu_n({\rm d}z)  - \int_K R_1(y,z) \mu_n({\rm d}z)\right|\\
			&\leq \int_K \left|   R_1(x,z) - R_1 (y,z)  \right| \mu_n({\rm d}z)\\
			&< \tau\cdot\mu_n(K).
		\end{aligned}
		\]
		For any $\varepsilon>0$, take $\tau=\varepsilon/ \sup_n \mu_n(K)$, we have completed the proof.
	\end{proof}
	\begin{proposition}\label{prop2.6}
		Assume $R_1(\cdot,\cdot)$ is continuous on $E\times E$. If $\mu \in {\sf S}_{L\!D}$,  then $R_1 (1_K\mu)\in C_b(E)$ for any compact set $K \subset E$. Moreover, if $\mu \in {\sf S}_{D_\infty}^1$, then $R_1 \mu \in C_b(E)$.
	\end{proposition}
	\begin{proof}
		For any $\mu \in {\sf S}_{L\!D}$ and any compact set $K$, since $R_1(x,y)$ is continuous on $E \times E$, then for any $\tau>0$ and $x\in E$, there exists $\delta>0$ such that for any $y\in E$ satisfying ${\sf d}(x,y)<\delta$, we have $\left| R_1(x,z)- R_1(y,z)  \right|<\tau$ for $z \in E$. Then,
		\[
		\begin{aligned}
			\left|  R_1 \mu^K(x) - R_1 \mu^K(y)  \right|&=\left|  \int_K R_1(x,z)\mu({\rm d}z )  -\int_K R_1(y,z)\mu({\rm d}z ) \right|\\
			&\leq \int_K \left|  R_1(x,z)-R_1(y,z)\right|\mu({\rm d}z ) \\
			&< \tau \mu(K).
		\end{aligned}
		\]
		For any $\varepsilon>0$, take $\tau=\varepsilon/\mu(K)$, we have $R_1 \mu^K \in C_b(E)$. For any $\mu \in {\sf S}_{D_\infty}^1$, we can see that $R_1\mu^K \in C_b(E)$ for any compact set $K \subset E$. Then for any $\varepsilon>0$, there exists $K\subset E$ such that $\Vert R_1(1_{K^c} \mu)\Vert_\infty<\varepsilon/3$. Since $R_1 \mu^K \in C_b(E)$, for any $x\in E$, then there exists $\delta>0$ such that for any $y\in E$ satisfying ${\sf d}(x,y)<\delta$, we have $\left|  R_1(x,z) - R_1(y,z)  \right|<\varepsilon/3\mu(K)$ for $z \in E$. Then
		\[
		\begin{aligned}
			\left|  R_1 \mu(x) - R_1 \mu(y)  \right|&=\left|  \int_E R_1(x,z)\mu({\rm d}z )  -\int_E R_1(y,z)\mu({\rm d}z ) \right|\\
			&\leq \int_K \left|  R_1(x,z)-R_1(y,z)\right|\mu({\rm d}z ) + \int_{K^c} \left|  R_1(x,z)-R_1(y,z)\right|\mu({\rm d}z )\\
			&<\varepsilon.
		\end{aligned}
		\]
	\end{proof}

	\section{Convergence of PCAFs}
	In this section, we introduce the definition of a positive continuous additive functional (PCAF).
	\begin{dy}
		A numerical function ${\sf A}_t(\omega)$, $t\geq0$, $\omega\in\Omega$ is a positive continuous additive functional if the following conditions hold:
		\begin{itemize}
			\item[(i)] For each $t\geq 0$, ${\sf A}_t(\cdot)$ is $\mathscr{F}_t$-measurable.
			\item[(ii)] There exists a set $\Lambda\in\mathscr{F}_\infty$ and $N\subset E$ with ${\rm Cap}(N)=0$ such that $\mathbb{P}_x(\Lambda)=1$ for any $x \in E\setminus N$ and $\theta_t \Lambda \subset \Lambda $ for any $t>0$. Moreover, for each $\omega \in \Lambda$, ${\sf A}_0(\omega)=0$, $|{\sf A}_t(\omega)|<\infty$ for any $t<\zeta(\omega)$, ${\sf A}_t(\omega)={\sf A}_{\zeta(\omega)}(\omega)$ for any $t\geq\zeta(\omega)$ and
			\[
			{\sf A}_{t+s}(\omega)={\sf A}_s(\omega)+{\sf A}_t(\theta_s\omega),\quad\text{for all}\quad t,s\geq0.
			\]
			\item[(iii)] For each $\omega \in \Lambda$, the map $t \mapsto {\sf A}_t(\omega)$ is nonnegative and continuous on $[0,+\infty)$.
		\end{itemize}
	\end{dy}
	The set of all PCAFs is denoted by $\mathbf{A}_c^+$. Two PCAFs ${\sf A}^{(1)}$ and ${\sf A}^{(2)}$, are said to be \textit{equivalent} if for each $t>0$, $\mathbb{P}_x({\sf A}_t^{(1)} = {\sf A}_t^{(2)})=1$ q.e. $x \in E$. If ${\sf A}^{(1)}$ and ${\sf A}^{(2)}$are equivalent, we write ${\sf A}^{(1)} \sim {\sf A}^{(2)}$. For any $\mu \in {\sf S}_{00}$, there exists PCAF ${\sf A^\mu}$ uniquely such that
	\begin{equation}\label{3.1}
	\mathbb{E}_x\left[ \int_{0}^{\infty} e^{-t} {\rm d}\ {\sf A}^\mu_t\right]=R_1 \mu(x),\quad \text{for all}\quad x\in E,
	\end{equation}
	see \cite[Theorem 5.1.6]{MR2778606}. In particular, when $\mu$ is only assumed to be a measure of finite energy integral, the above equality does not necessarily hold for all $x\in E$, but rather for q.e. $x \in E$. The following theorems describes the concrete relationship between ${\sf S}$ and $\mathbf{A}_c^+$.
	\begin{theorem}[\cite{MR2778606} Theorem 5.1.3]
		 The quotient space of $\mathbf{A}_c^+$ under the equivalence relation and the family ${\sf S}$ are in one-to-one correspondence under the following relation: For $\mu \in {\sf S}$ and ${\sf A}\in\mathbf{A}_c^+$,
		 \begin{equation}\label{3.2}
		 	\lim_{t \downarrow 0} \frac{1}{t}\int_E \mathbb{E}_x\left[  \int_{0}^{t} f(X_s)\ {\rm d}{\sf A}_s  \right]h(x)\mathfrak{m}({\rm d}x)=\int_E f(x) h(x)\mu({\rm d}x)
		 \end{equation}
		 where $h$ is any $\gamma$-excessive function and $f \in \mathscr{B}_+(E)$.
	\end{theorem}
	The equation (\ref{3.2}) is called the \textit{Revuz correspondence}. In this paper, we call it the \textit{Revuz map} by viewing this correspondence from ${\sf S}$ to $\mathbf{A}_c^+$. Next, we recall a theorem concerning the convergence of PCAF. 
	\begin{theorem}[\cite{Nishimori_Tomisaki_Tsuchida_Uemura_2025} Theorem 4.1]\label{thm3.2}
		Let ${\sf A}^n$, ${\sf A}\in\mathbf{A}_c^+$. Denote by $\mu_n$ and $\mu$ their Revuz measures in ${\sf S}$ for $n\geq 1$. Assume that all $\mu_n$ and $\mu$ belong to ${\sf S}_0$. If $\displaystyle \lim_{n \to \infty} \rho(\mu_n,\mu)=0$, then there exists a subsequence $\{n_k\}$ of $\{ n\}$ such that 
		\[
		\mathbb{P}_x\left(   \lim_{n_k \to \infty}  {\sf A}^{n_k}_t= {\sf A}_t  \text{ locally uniformly in }t \text{ on }[0,+\infty) \right)=1\quad \text{q.e.}\quad x\in E.
		\]
	\end{theorem}
	\begin{theorem}[\cite{MR4968118} Theorem 3.7]\label{thm3.3}
		For $\mu_n,\mu \in {\sf S}_0$, let ${\sf A}^n,{\sf A} \in \mathbf{A}_c^+$ be the corresponding of PCAFs, respecetively. If $\rho(\mu_n,\mu)\to 0$ as $n\to\infty$, then ${\sf A}^n$ converges to ${\sf A}$ with the local uniform topology in $L^1(\mathbb{P}_x)$ for q.e. $x \in E$, that is, for any $T>0$ and q.e. $x \in E$,
		\[
		\lim_{n \to \infty}\mathbb{E}_x\left[ \sup_{0\leq t \leq T} \left| {\sf A}^n_t - {\sf A}_t  \right|  \right]=0.
		\]
	\end{theorem}
	\begin{remark}\label{rem3.1}
		From Theorem~\ref{thm3.3}, for any $\nu \in {\sf S}_{00}$ and $T,\delta>0$, there exists positive constant $C_\nu$ depending on $\nu$,
		\[
		\mathbb{P}_\nu \left(    \sup_{0\leq t \leq T}\left|  {\sf A}_t^n -{\sf A}_t \right|\geq \delta   \right)\leq C_\nu \sqrt{\mathcal{E}_1(U_1 \mu_n-U_1\mu, U_1 \mu_n-U_1\mu)}.
		\]
		So, by using Borel-Cantelli's lemma, we can take a subsequence $\{n_k\} \subset \{n\}$ independent of $\nu$ such that 
		\[
		\mathbb{P}_x \left(  \lim_{n_k\to\infty}  \sup_{0\leq t \leq T}\left|  {\sf A}_t^{n_k}-{\sf A}_t \right|=0 \text{ for any T}  \right)=1,
		\]
		for any $T>0$ and q.e. $x\in E$ (see \cite[Remark 3.9]{MR4968118}).
	\end{remark}
	Through the  equation (\ref{3.1}), the conclusion of the preceding theorem can be extended to all points $x \in E$. By revisiting its proof, we obtain the following two propositions.
    \begin{lemma}[cf. \cite{MR2778606} Lemma 5.4.1]\label{lem3.1}
		For $\mu \in {\sf S}_{00}$ corresponding to ${\sf A} \in \mathbf{A}_c^+$ and $\alpha>0$,  there exists an martingale additive functional $M^{[ R_\alpha\mu]}$ such that, for $t \geq 0$ and all $x \in E$,
		\[
		R_\alpha \mu(X_t)- R_\alpha \mu(X_0)= M_t^{[R_1\mu]}+\alpha \int_0^t R_\alpha \mu(X_s)\ {\rm d}s -{\sf A}_t,\quad\mathbb{P}_x\text{-a.e..}
		\]
	\end{lemma}
	\begin{proof}
		For any $\alpha>0$ and $t \geq 0$, let $u:=R_\alpha \mu$ for $\mu \in {\sf S}_{00}$. Define $N_t$ by
		\[
		N_t^{[u]}:= e^{-\alpha t}u(X_t)+\int_{0}^{t}e^{-\alpha r} {\rm d}{\sf A}_r.
		\]
		First, we will show that $N^{[u]}_t$ is $\mathbb{P}_x$-martingale for all $x \in E$. For any  fixed $0 \leq s \leq t$, set $h=s-t$, then
		\[
		\begin{aligned}
			\mathbb{E}_x\left[\left.    e^{-\alpha t} u (X_t) +\int_s^t e^{-\alpha r}\ {\rm d}{\sf A}_r \ \right|\ \mathscr{F}_s \right]&=\mathbb{E}_x\left[\left.    e^{-\alpha(s+h)} u(X_{s+h}) +e^{-\alpha s}\int_0^h e^{-\alpha r}\ {\rm d}{\sf A}_r\circ\theta_s  \ \right|\ \mathscr{F}_s \right]\\
			&=e^{-\alpha s}\mathbb{E}_{X_s}\left[   e^{-\alpha h} u(X_h)+\int_0^h e^{-\alpha r}\ {\rm d}{\sf A}_r  \right]\\
			&=e^{-\alpha s}\mathbb{E}_{X_s}\left[   e^{-\alpha h} \mathbb{E}_{X_h}\left[ \int_{0}^{\infty}e^{-\alpha r}\ {\rm d}{\sf A}_r \right]+\int_0^h e^{-\alpha r}\ {\rm d}{\sf A}_r  \right]\\
			&=e^{-\alpha s}\mathbb{E}_{X_s}\left[    \mathbb{E}_{X_s}\left[\left. \int_{h}^{\infty}e^{-\alpha r}\ {\rm d}{\sf A}_r\ \right|\ \mathscr{F}_h\right]+\int_0^h e^{-\alpha r}\ {\rm d}{\sf A}_r  \right]\\
			&=e^{-\alpha s}\mathbb{E}_{X_s}\left[   \int_h^\infty e^{-\alpha r}\ {\rm d}{\sf A}_r  +\int_0^h e^{-\alpha r}\ {\rm d}{\sf A}_r  \right]\\
			&=e^{-\alpha s}u(X_s).
		\end{aligned}
		\]
		Hence, 
		\[
		\begin{aligned}
			\mathbb{E}_x\left[\left.N^{[u]}_t \ \right|\  \mathscr{F}_s\right]&=\mathbb{E}_x\left[\left.   e^{-\alpha t}u(X_t)+\int_{0}^{s}e^{-\alpha r} {\rm d}{\sf A}_r+   \int_{s}^{t}e^{-\alpha r} {\rm d}{\sf A}_r \ \right|\ \mathscr{F}_s  \right]\\
			&=\int_{0}^{s}e^{-\alpha r} {\rm d}{\sf A}_r+ \mathbb{E}_x\left[\left.   e^{-\alpha t}u(X_t)+   \int_{s}^{t}e^{-\alpha r} {\rm d}{\sf A}_r  \ \right|\ \mathscr{F}_s \right]\\
			&=\int_{0}^{s}e^{-\alpha r} {\rm d}{\sf A}_r+e^{-\alpha s}u(X_s)\\
			&=N_s^{[u]}.
		\end{aligned}
		\]
		This is show that $N^{[u]}_t$ is $\mathbb{P}_x$-martingale for all $x \in E$. Let $\displaystyle M_t^{[u]}:=\int_0^t e^{\alpha r}{\rm d}N^{[u]}_r$ be a $\mathbb{P}_x$-martingale for all $x \in E$. Then,
		\[
		{\rm d}N^{[u]}_t={\rm d}\left( e^{-\alpha t}u(X_t) \right)+e^{-\alpha t}{\rm d}{\sf A}_r=e^{-\alpha t}{\rm d}u(X_t)-\alpha e^{-\alpha t}u(X_t){\rm d}t+e^{-\alpha t}{\rm d}{\sf A}_r.
		\]
		Multiply both sides by $e^{\alpha t}$, we can have
		\[
		e^{\alpha t}{\rm d}N^{[u]}_t={\rm d}u(X_t)-\alpha u(X_t){\rm d}t+{\rm d}{\sf A}_r
		\]
		Since the left side of the equation above equals ${\rm d}M_t^{[u]}$, integrating both sides yields
		\[
		M_t^{[u]}=u(X_t)-u(X_0)-\alpha\int_0^t u(X_s)\ {\rm d}s +{\sf A}_t.
		\]
		For any $s,t \geq 0$, $\alpha>0$ and all $x \in E$,
		\[
		\begin{aligned}
			M_t^{[u]}+M_s^{[u]}\circ\theta_t&=u(X_t)-u(X_0)+(u(X_s))\circ \theta_t-(u(X_0))\circ\theta_t\\
			&-\alpha\int_0^t u(X_r)\ {\rm d}r-\left(\int_0^s u(X_r)\ {\rm d}r\right)\circ\theta_t+{\sf A}_t +{\sf A}_s\circ\theta_s\\
			&=u(X_{t+s})- u(X_0) -\alpha \int_{0}^{t+s} u(X_r)\ {\rm d}r +{\sf A}_{t+s}\\
			&=M_{t+s}^{[u]}\quad \mathbb{P}_x\text{a.s.}
		\end{aligned}
		\]
		Therefore, $M_t^{[u]}$ is an martingale additive functional.
	\end{proof}

	\begin{theorem}\label{thm3.4}
		For $\mu_n,\mu \in \mathcal{S}_{00}$, let ${\sf A}^n,{\sf A}\in\mathbf{A}_c^+$ be the corresponding PCAFs, respectively. If ,$R_1 \mu_n $ is converges to $R_1 \mu$ uniformly on $E$, then for all $x \in E$ and any $T>0$,
		\[
		\lim_{n\to \infty} \mathbb{E}_x\left[   \sup_{0\leq t \leq T} \left|  {\sf A}_t^n -{\sf A}_t\right|  \right]=0.
		\]
	\end{theorem}
	\begin{proof}
		From Lemma~\ref{lem3.1}, there exists $M^{[R_1\mu]}$ such that for all $x \in E$ and $t \geq 0$,
		\[
		R_1 \mu(X_t)- R_1 \mu(X_0)= M_t^{[R_1\mu]}+ \int_0^t R_1 \mu(X_s)\ {\rm d}s -{\sf A}_t,\quad\mathbb{P}_x\text{-a.s.}
		\]
		Let $\Delta^n:=R_1 (\mu_n- \mu)$ and $M^{[\Delta^n]}_t:=M_t^{[R_1\mu_n]}-M_t^{[R_1\mu]}$ for any $t\geq 0$. Then, we have
		\[
		{\sf A}^n_t - {\sf A}_t=\Delta^n(X_t) - \Delta^n(X_0) -\int_0^t \Delta^n(X_s)\ {\rm d}s-M^{[\Delta^n]}_t.
		\]
		Hence, 
		\[	
			\begin{aligned}
				\mathbb{E}_x\left[\sup_{0\leq t \leq T} \left|  {\sf A}^n_t - {\sf A}_t\right| \right]&\leq \mathbb{E}_x\left[\sup_{0\leq t \leq T} |\Delta^n(X_t)|\right] +|\Delta^n(X_0)|+\mathbb{E}_x\left[\int_0^T |\Delta^n(X_s)|\ {\rm d}s\right]+\mathbb{E}_x\left[\sup_{0\leq t \leq T} \left|M_t^{[\Delta^n]}\right|\right]\\
				&=:{\rm I_1+I_2+I_3+I_4}.
			\end{aligned}
		\]
		Next, we'll check these four items one by one. Let's start with the first one.
		\[
		{\rm I_1}=\mathbb{E}_x\left[ \sup_{0\leq t \leq T} \left|  R_1 \mu(X_t)- R_1 \mu_n(X_t)\right| \right]\leq \Vert R_1 \mu - R_1 \mu_n \Vert_\infty\to0\text{ as }n \to \infty.
		\]
		Silimlarly,
		\[
		{\rm I_2}= |R_1 \mu(X_0) - R_1 \mu_n(X_0)|  \leq \Vert R_1 \mu - R_1 \mu_n \Vert_\infty\to0\text{ as }n \to \infty.
		\]
		Through the definition of the transition kernel and the Markov property, we can obtain
		\[
		\begin{aligned}
			{\rm I_3}&=\int_0^T \mathbb{E}_x\left[\left|\Delta^n_s\right|\right]\ {\rm d}s\\
			&=\int_0^T P_s\left( |R_1 \mu_n(X_s) - R_1 \mu(X_s) | \right)\ {\rm d}s\\
			&\leq \Vert R_1 \mu_n - R_1\mu \Vert_\infty \int_0^T P_s1(X_s)\ {\rm d}s\\
			&\leq T\cdot \Vert R_1 \mu_n - R_1\mu \Vert_\infty \to 0\quad \text{as}\quad n\to \infty.
		\end{aligned}
		\]
		Finally, From Doob's inequality and \cite[Theorem 5.2.3]{MR2778606}, we have
		\[
		\begin{aligned}
			\mathbb{E}_x\left[\sup_{0\leq t \leq T} \left(M^{[\Delta^n]}_t\right)^2\right]&\leq 4\cdot\mathbb{E}_x\left[  \left(M_T^{[\Delta^n]}\right)^2 \right]\\
			&=4\cdot\mathbb{E}_x\left[  \langle  M^{[\Delta^n]} \rangle_T   \right]\\
			&\leq 4e^{T}\mathbb{E}_x\left[  \int_0^\infty e^{-t}\ {\rm d} \langle M^{[\Delta^n]} \rangle_t \right]\\
			&=4e^{T}\int_E R_1(x,y)\ \mu_{ \langle\Delta^n \rangle}({\rm d}y),
		\end{aligned}
		\]
		where $\mu_{ \langle\Delta^n \rangle}$ is energy measure. Since
		\[
		\mu_{ \langle\Delta^n \rangle}(E) =2\mathcal{E}(\Delta^n,\Delta^n)\leq2\mathcal{E}_1(\Delta^n,\Delta^n)\leq2\Vert R_1\mu_n-R_1\mu\Vert_\infty\cdot\left( \mu_n+\mu \right)(E),
		\]
		then, we have
		\[
		\begin{aligned}
			\int_E R_1(x,y)\ \mu_{ \langle\Delta^n \rangle}({\rm d}y) &\leq 2\Vert R_1 \mu_n-R_1\mu \Vert_\infty \int_E R_1(x,y)\ (\mu_n+\mu)({\rm d}y)\\
			&\leq2 \Vert R_1 \mu_n-R_1\mu \Vert_\infty\cdot\left( R_1\mu_n+R_1\mu  \right)\\
			&\to 0\quad\text{as} \quad n\to\infty.
		\end{aligned}
		\]
		Therefore, $I_4 \to 0$ as $n\to\infty$, so we have completed the proof of the theorem.

	\end{proof}

	\begin{remark}\label{rem3.2}
		From Theorem~\ref{thm3.4} and Markov's inequality, for any $T,\delta>0$, there exists a constant $C_{T,\delta}$ such that 
		\[
		\mathbb{P}_x \left(    \sup_{0\leq t \leq T}\left|  {\sf A}_t^n -{\sf A}_t \right|\geq \delta   \right)\leq C_{T,\delta}\cdot\Vert R_1\mu_n-R_1\mu\Vert_\infty.
		\]
		So, by using Borel-Cantelli's lemma, we can take a subsequence $\{n_k\} \subset \{n\}$ independent of $\nu$ such that 
		\[
		\mathbb{P}_x \left(  \lim_{n_k\to\infty}  \sup_{0\leq t \leq T}\left|  {\sf A}_t^{n_k}-{\sf A}_t \right| =0\text{ for any T}  \right)=1,
		\]
		for any $T>0$ and every $x\in E$.
	\end{remark}

	\section{Main Theorem}
	\begin{theorem}\label{thm4.1}
		Let ${\sf A}^n,{\sf A}\in\mathbf{A}_c^+$. Denote by $\mu_n$, $\mu$ their Revuz measures in ${\sf S}_0$ for each $n$. Assume that $\mu \in {\sf S}_{D}$ and $\{\mu_n\}$ is of uniformly Dynkin. If
		\begin{equation}\label{4.1}
		\mathbb{P}_x\left(   \lim_{n \to \infty}  {\sf A}^{n}_t= {\sf A}_t  \text{ locally uniformly in }t \text{ on }[0,+\infty) \right)=1\quad \text{q.e.}\quad x\in E,
		\end{equation}
		then $\lim_{n \to \infty} R_1 \mu_n(x)= R_1 \mu(x)$ for q.e. $x \in E$
	\end{theorem}
	\begin{proof}
		For any $T>0$, from integration by parts  formula, we have
		\[
		\int_0^T e^{-t} {\rm d}{\sf A}^n_t=e^{-T}{\sf A}^n_T +\int_0^T e^{-t}{\sf A}_t^n \ {\rm d}t \quad\text{and}\quad \int_0^T e^{-t} {\rm d}{\sf A}_t=e^{-T}{\sf A}_T +\int_0^T e^{-t}{\sf A}_t \ {\rm d}t. 
		\]
		By assumption, for q.e. $x\in E$,
		\[
		\sup_{0\leq t\leq T}\left| {\sf A}_t^n - {\sf A}_t  \right|\to 0\quad\text{as}\quad n\to\infty\quad \mathbb{P}_x\text{-a.s.}
		\]
		Therefore,
		\[
		\begin{aligned}
			\left|  \int_0^T e^{-t}\ {\rm d}{\sf A}_t^n -   \int_0^T e^{-t}\ {\rm d}{\sf A}_t\right|&=\left| e^{-T}({\sf A}_T^n -{\sf A}_T) +\int_0^T e^{-t} \left({\sf A}_t^n-{\sf A}_t\right) \ {\rm d}t  \right|\\
			&\leq 2\sup_{0\leq t\leq T}\left| {\sf A}_t^n -{\sf A}_t  \right|\\
			&\to 0\quad\text{as}\quad n \to \infty \quad\mathbb{P}_x\text{-a.s.}
		\end{aligned}
		\]
		We obtain, for every fixed $T>0$ and q.e. $x \in E$,
		\[
		\lim_{n \to \infty} \int_0^T e^{-t} {\rm d}{\sf A}^n_t =\int_0^T e^{-t} {\rm d}{\sf A}_t\quad\mathbb{P}_x\text{-a.s.}
		\]
		In fact, that $\displaystyle \mathbb{E}_x\left[  \int_{0}^{\infty}  e^{-t}{\rm d}{\sf A}_t \right]$ and $\displaystyle \mathbb{E}_x\left[  \int_{0}^{\infty}  e^{-t}{\rm d}{\sf A}^n_t \right]$ are quasi continuous modification of the 1-potential $U_1\mu$  for each $n$, respectively. Set $M:=  \sup_n\Vert R_1 \mu_n \Vert_\infty$, we have that
		\[
		\begin{aligned}
			\mathbb{E}_x\left[ \left(\int_0^T e^{-t}\ {\rm d}{\sf A}_t^n\right)^2 \right]&\leq  \mathbb{E}_x\left[  \left(\int_0^\infty e^{-t} \ {\rm d}{\sf A}_t^n\right)^2  \right]\\
			&=2 \mathbb{E}_x\left[ \int_0^\infty e^{-2t}\  \mathbb{E}_{X_t} \left[  \int_0^\infty e^{-s} \ {\rm d}{\sf A}_s^n \right] {\rm d}{\sf A}_t^n \right]\\
			&\leq 2 M^2.
		\end{aligned}
		\]
		Hence $\int_{0}^{T} e^{-t}\ {\rm d}{\sf A}_t^n$  uniformly integrable with respect to $\mathbb{P}_x$. By Vitali’s convergence theorem, we have
		\[
		\lim_{n \to \infty}\mathbb{E}_x\left[ \int_{0}^{T} e^{-t}\ {\rm d}{\sf A}_t^n  \right]=\mathbb{E}_x\left[ \int_{0}^{T} e^{-t}\ {\rm d}{\sf A}_t  \right]\quad \text{for q.e. }x\in E.
		\]
		Consider, for q.e. $x \in E$, from the strong Markov property, we have
		\[
			\mathbb{E}_x\left[ \int_{T}^{\infty} e^{-t}\ {\rm d}{\sf A}_t  \right]=\mathbb{E}_x\left.\left[e^{-T} \mathbb{E}_{X_T}   \left[   \int_{0}^{\infty} e^{-t}\ {\rm d}{\sf A}_t   \right]\ \right|\ \mathscr{F}_T\right]\leq e^{-T} \Vert R_1 \mu \Vert_\infty.
		\]
		Hence, for q.e. $x \in E$,
		\[
		\mathbb{E}_x\left[ \int_{T}^{\infty} e^{-t}\ {\rm d}{\sf A}_t^n  \right]\leq e^{-T} M\quad\text{and}\quad \mathbb{E}_x\left[ \int_{T}^{\infty} e^{-t}\ {\rm d}{\sf A}_t  \right]\leq e^{-T} \Vert R_1 \mu \Vert_\infty.
		\]
		Therefore,   $\mathbb{E}_x\left[ \int_{T}^{\infty} e^{-t}\ {\rm d}{\sf A}_t^n  \right] \to 0$ as $T \to \infty$ for each $n$ and $\mathbb{E}_x\left[ \int_{T}^{\infty} e^{-t}\ {\rm d}{\sf A}_t  \right] \to 0$ as $T \to \infty$ . Then, for any $\varepsilon>0$, there exists $T_0>0$ such that for any $n\in\mathbb{N}$,
		\[
		\mathbb{E}_x\left[ \int_{T_0}^{\infty} e^{-t}\ {\rm d}{\sf A}_t^n  \right] <\frac{\varepsilon}{3}\quad\text{and}\quad\mathbb{E}_x\left[ \int_{T_0}^{\infty} e^{-t}\ {\rm d}{\sf A}_t  \right] <\frac{\varepsilon}{3}.
		\]
		For any $\varepsilon>0$ and q.e. $x \in E$, there exists $N\in\mathbb{N}$ such that for any  $n\geq N$,
		\[
		\left| \mathbb{E}_x\left[ \int_{0}^{T} e^{-t}\ {\rm d}{\sf A}_t  \right] - \mathbb{E}_x\left[ \int_{0}^{T} e^{-t}\ {\rm d}{\sf A}_t^n  \right]\right|<\frac{\varepsilon}{3}.
		\]
		Therefore,
		\[
		\begin{aligned}
			&\left| \mathbb{E}_x\left[ \int_{0}^{\infty} e^{-t}\ {\rm d}{\sf A}_t^n  \right] -\mathbb{E}_x\left[ \int_{0}^{\infty} e^{-t}\ {\rm d}{\sf A}_t  \right]  \right|\\
			\leq&\left| \mathbb{E}_x\left[ \int_{0}^{T_0} e^{-t}\ {\rm d}{\sf A}_t^n  \right] -\mathbb{E}_x\left[ \int_{0}^{T_0} e^{-t}\ {\rm d}{\sf A}_t  \right]  \right|+\mathbb{E}_x\left[ \int_{T_0}^{\infty} e^{-t}\ {\rm d}{\sf A}_t^n\right]+\mathbb{E}_x\left[ \int_{T_0}^{\infty} e^{-t}\ {\rm d}{\sf A}_t\right]\\
			<& \frac{\varepsilon}{3}+\frac{\varepsilon}{3}+\frac{\varepsilon}{3}=\varepsilon.
		\end{aligned}
		\]
		Therefore, $\lim_{n \to \infty} R_1 \mu_n(x)= R_1 \mu(x)$ for q.e. $x \in E$.
	\end{proof}
	
	\begin{remark}\label{rem4.1}
		If, in addition, $\mu_n,\mu \in {\sf S}_{00}$ and the convergence condition in Theorem~\ref{thm4.1} holds for every $x\in E$ (not merely quasi-everywhere), then for the conclusion can be strengthened to 
		\[
		\lim_{n\to\infty} R_1\mu_n(x)=R_1\mu(x)\quad \text{for all}\quad x\in E.
		\]
		Indeed, from (\ref{3.1}), for $\mu\in{\sf S}_{00}$, the function $R_1\mu$ is bounded and, by the Revuz correspondence, equals the expectation $\displaystyle\mathbb{E}_x\left[\int_0^\infty e^{-t}{\rm d}{\sf A}_t\right]$ for every $x\in E$. The pointwise convergence of the PCAFs (for all $x$) and the boundedness of the potentials allow us to apply the dominated convergence theorem to pass to the limit in the expectation, yielding the desired pointwise convergence of $R_1\mu_n$ to $R_1\mu$ for every $x\in E$.
	\end{remark}

	\begin{theorem}\label{thm4.2}
		Let ${\sf A}^n,{\sf A}\in\mathbf{A}_c^+$. Denote by $\mu_n$, $\mu$ their Revuz measures in ${\sf S}$ for each $n$. Assume that $\mu \in {\sf S}_{L\!D}$ and $\{\mu_n\}$ is of uniformly local Dynkin. For any compact set $K \subset E$, if the resolvent kernel $R_1(x,y)$ is continuous on $E\times E$. Then, $R_1 \mu^K,R_1\mu_n^K \in C_b(E)$ and the following conditions are equivalent:
		\begin{itemize}
			\item[{\rm (i)}] $\mu_n$  converges to $\mu$ weakly on $K$.
			\item[{\rm (ii)}] $U_1 \mu^K_n$  converges to $U_1 \mu^K$ strongly in $\mathcal{E}_1$-norm.
			\item[{\rm (iii)}] $\displaystyle \lim_{n\to \infty}\mathbb{E}_x\left[ \sup_{0\leq t\leq T} \left|  (1_K{\sf A}^n)_t - (1_K{\sf A})_t \right|  \right] =0  \text{ for any }T>0\text{ and any }\  x\in E.$
			\item[{\rm (iv)}] $R_1 \mu_n^K$  converges to $R_1 \mu^K$ uniformly on $E$.
		\end{itemize}
	\end{theorem}
	\begin{proof}
		The proof of continuity follows directly from the Proposition~\ref{prop2.6}. ${\rm (i)} \Rightarrow {\rm (iv)}$: Since $R_1 (x,y)$ is continuous on $E\times E$, then for any $x \in E$
		\[
		\lim_{n \to \infty} R_1 \mu_n^K(x)=\int_K R_1(x,y)\ \mu_n({\rm d}y) =\int_K R_1(x,y)\ \mu({\rm d}y) =R_1 \mu^K(x).
		\]
		Form Proposition~\ref{prop2.5}, by Arzel\`a–Ascoli Theorem, we have $R_1 \mu^K_n$ converges to $R_1 \mu^K$ uniformly on $ K$. For each $x \in E$, from the strong Markov property, we have
		\[
		\begin{aligned}
			R_1 \mu^K(x)&=\mathbb{E}_x\left[ \int_{0}^{\infty} e^{-t} \ {\rm d}(1_K {\sf A})_t  \right]\\
			&=\mathbb{E}_x\left[ \int_{0}^{\infty} e^{-t} 1_K(X_t)\ {\rm d}{\sf A}_t  \right]\\
			&=\mathbb{E}_x\left[ \int_{\sigma_K}^{\infty} e^{-t} \ {\rm d}{\sf A}_t  \right]\\
			&=\mathbb{E}_x \left[ e^{-\sigma_K} \mathbb{E}_{X_{\sigma_K}} \left[ \int_0^\infty e^{-t} \ {\rm d}{\sf A}_t   \right]\cdot 1_{\{\sigma_k<\infty\}} \right],
		\end{aligned}
		\]
		where $\sigma_K:=\inf\{t>0\mid X_t \in K \}$ be the first hitting time to $K$. Hence,
		\[
		\begin{aligned}
			\left| R_1 \mu^K(x) - R_1 \mu_n(x) \right|&=\left|\mathbb{E}_x \left[ e^{-\sigma_K} \mathbb{E}_{X_{\sigma_K}} \left[ \int_0^\infty e^{-t} \ {\rm d}{\sf A}_t  - \int_0^\infty e^{-t} \ {\rm d}{\sf A}^n_t  \right]\cdot 1_{\{\sigma_k<\infty\}} \right]\right|\\
			&= \left| \mathbb{E}_x\left[ e^{-\sigma_K} \left( R_1(1_K \mu)(X_{\sigma_K}) - R_1(1_K \mu_n)(X_{\sigma_K})  \right) \cdot 1_{\{\sigma_K<\infty\}} \right]  \right|\\
			&\leq \Vert R_1 \mu^K - R_1 \mu^K_n \Vert_{K,\infty} .
		\end{aligned}
		\]
		Therefore, $R_1 \mu^K_n$ converges to $R_1 \mu^K$ uniformly on $E$.
		\par ${\rm (iv)} \Rightarrow {\rm (ii)}$: Let $M_1:= \mu(K)$ and $M_2:= \sup_n \mu_n(K)$.
		\[
		\begin{aligned}
			&\left| \mathcal{E}_1(U_1 \mu_n^K- U_1 \mu^K, U_1 \mu_n^K- U_1 \mu^K) \right|\\
			\leq&\left| \int_K R_1 \mu_n^K - R_1\mu^K  \ {\rm d}\mu_n \right| +\left| \int_K R_1 \mu_n^K - R_1\mu^K  \ {\rm d}\mu \right|\\
			\leq&(M_1 + M_2)\sup_{x \in K} \left|R_1 \mu_n^K(x)-R_1 \mu^K(x)\right|\\
			\to& 0 \quad \text{as}\quad n \to \infty.
		\end{aligned}
		\]
		Therefore, $U_1 \mu^K_n$ converges to $U_1 \mu^K$ strongly in $\mathcal{E}_1$-norm.
		\par Form Propositon~\ref{prop 2.1} and Theorem~\ref{thm3.4}, ${\rm (ii)}\Rightarrow {\rm (i)}$ and ${\rm (iv)} \Rightarrow {\rm (iii)}$  is obvious. ${\rm (iii)}\Rightarrow {\rm (iv)}$: Let ${\sf B}^n_t:=(1_K{\sf A}^n)_t$ and ${\sf B}_t:=(1_K{\sf A})_t$ be are PCAFs. For any $T>0$, from integration by parts  formula, we have
		\[
		\int_0^T e^{-t} {\rm d}{\sf B}^n_t=e^{-T}{\sf B}^n_T +\int_0^T e^{-t}{\sf B}_t^n \ {\rm d}t \quad\text{and}\quad \int_0^T e^{-t} {\rm d}{\sf B}_t=e^{-T}{\sf B}_T +\int_0^T e^{-t}{\sf B}_t \ {\rm d}t. 
		\]
		By assumption, for all $x\in E$,
		\[
		\begin{aligned}
			\left| \mathbb{E}_x\left[ \int_0^T e^{-t}\ {\rm d}{\sf B}_t^n\right]  -\mathbb{E}_x \left[ \int_0^T e^{-t}\ {\rm d}{\sf B}_t\right] \right|&\leq \mathbb{E}_x\left[ \left| e^{-T}({\sf B}_T^n -{\sf B}_T) +\int_0^T e^{-t} \left({\sf B}_t^n-{\sf B}_t\right) \ {\rm d}t  \right|  \right]\\
			&\leq 2\cdot\mathbb{E}_x \left[ \sup_{0\leq t\leq T}\left| {\sf B}_t^n -{\sf B}_t  \right|\right]\\
			&\to 0 \quad\text{as}\quad n\to\infty.
		\end{aligned}
		\]
		Hence, we have
		\[
		\lim_{n \to \infty}\mathbb{E}_x\left[ \int_{0}^{T} e^{-t}\ {\rm d}{\sf B}_t^n  \right]=\mathbb{E}_x\left[ \int_{0}^{T} e^{-t}\ {\rm d}{\sf B}_t  \right]\quad \text{for q.e. }x\in E.
		\]
		Consider, for all $x \in E$
		\[
		\mathbb{E}_x\left[ \int_{T}^{\infty} e^{-t}\ {\rm d}{\sf B}_t^n  \right]\leq e^{-T} \cdot\sup_{n \in \mathbb{N}}\Vert R_1 \mu_n^K \Vert_\infty\quad\text{and}\quad \mathbb{E}_x\left[ \int_{T}^{\infty} e^{-t}\ {\rm d}{\sf B}_t  \right]\leq e^{-T} \Vert R_1 \mu^K \Vert_\infty.
		\]
		Therefore,   $\mathbb{E}_x\left[ \int_{T}^{\infty} e^{-t}\ {\rm d}{\sf B}_t^n  \right] \to 0$ as $T \to \infty$ for each $n$ and $\mathbb{E}_x\left[ \int_{T}^{\infty} e^{-t}\ {\rm d}{\sf B}_t  \right] \to 0$ as $T \to \infty$ . Then, for any $\varepsilon>0$, there exists $T_0>0$ such that for any $n\in\mathbb{N}$,
		\[
		\mathbb{E}_x\left[ \int_{T_0}^{\infty} e^{-t}\ {\rm d}{\sf B}_t^n  \right] <\frac{\varepsilon}{3}\quad\text{and}\quad\mathbb{E}_x\left[ \int_{T_0}^{\infty} e^{-t}\ {\rm d}{\sf B}_t  \right] <\frac{\varepsilon}{3}.
		\]
		For any $\varepsilon>0$ and all  $x \in E$, there exists $N\in\mathbb{N}$ such that for any  $n\geq N$,
		\[
		\left| \mathbb{E}_x\left[ \int_{0}^{T} e^{-t}\ {\rm d}{\sf B}_t  \right] - \mathbb{E}_x\left[ \int_{0}^{T} e^{-t}\ {\rm d}{\sf B}_t^n  \right]\right|<\frac{\varepsilon}{3}.
		\]
		Therefore,
		\[
		\begin{aligned}
			&\left| \mathbb{E}_x\left[ \int_{0}^{\infty} e^{-t}\ {\rm d}{\sf B}_t^n  \right] -\mathbb{E}_x\left[ \int_{0}^{\infty} e^{-t}\ {\rm d}{\sf B}_t  \right]  \right|\\
			\leq&\left| \mathbb{E}_x\left[ \int_{0}^{T_0} e^{-t}\ {\rm d}{\sf B}_t^n  \right] -\mathbb{E}_x\left[ \int_{0}^{T_0} e^{-t}\ {\rm d}{\sf B}_t  \right]  \right|+\mathbb{E}_x\left[ \int_{T_0}^{\infty} e^{-t}\ {\rm d}{\sf B}_t^n\right]+\mathbb{E}_x\left[ \int_{T_0}^{\infty} e^{-t}\ {\rm d}{\sf B}_t\right]\\
			<& \frac{\varepsilon}{3}+\frac{\varepsilon}{3}+\frac{\varepsilon}{3}=\varepsilon.
		\end{aligned}
		\]
		Therefore, $\lim_{n \to \infty} R_1 \mu_n^K(x)= R_1 \mu^K(x)$ for all $x \in E$. From Proposition~\ref{prop2.5}, by  Arzel\`a–Ascoli Theorem, we can get the conclusion. 
	\end{proof}

	\begin{remark}
		In connection with the equivalence established in Theorem~\ref{thm4.2}, we note the following additional observation concerning almost sure convergence of the full sequence of PCAFs. Suppose that in addition to the assumptions of Theorem~\ref{thm4.2}, the energy condition
		\begin{equation}\label{4.2}
		\sum_{n=1}^{\infty} \mathcal{E}_1\big(U_1\mu_n^K,\; U_1\mu_n^K\big)<\infty
		\end{equation}
		holds for a given compact set $K\subset E$. Then, by virtue of Remark~\ref{rem3.2}, the uniform convergence of $R_1\mu_n$ can be strengthened to
		\[
		\mathbb{P}_x\left(\lim_{n\to\infty} (1_K\mathsf{A}^n)_t = (1_K\mathsf{A})_t\text{ locally uniformly in } t \text{ on } [0,+\infty)\right)=1\quad\text{for all } x\in E.
		\]
		Consequently, under this energy summability condition (\ref{4.2}), form Remark~\ref{rem4.1}, the above almost sure convergence (for all $x\in E$) is also equivalent to the four conditions (i)--(iv) of Theorem~\ref{thm4.2}.\\
		If the summability condition (\ref{4.2}) is not imposed, we can still conclude the existence of a subsequence $\{n_k\}$ such that the almost sure convergence of PCAFs along this subsequence (for all $x\in E$) is equivalent to the corresponding subsequential versions of the four conditions (i)--(iv). In other words, the five conditions (i)--(iv) and the almost sure convergence along the subsequence $\{n_k\}$ are all equivalent on this subsequence.
	\end{remark}

	\begin{theorem}\label{thm4.3}
		Let ${\sf A}^n,{\sf A}\in\mathbf{A}_c^+$. Denote by $\mu_n$, $\mu$ their Revuz measures in ${\sf S}_0$ for each $n$. Assume that $\mu\in {\sf S}^1_{G\!D}$ and $\{\mu_n\}$ is 1-order uniformly  Green-tight measure of Dynkin. If the resolvent kernel $R_1(x,y)$ is continuous on $E\times E$. Then $R_1 \mu,R_1\mu_n \in C_b(E)$ and the following conditions are equivalent:
		\begin{itemize}
			\item[{\rm (i)}] $\mu_n$ converges to $\mu$ vaguely on $E$.
			\item[{\rm (ii)}] $U_1 \mu_n$ converges to $U_1 \mu$ strongly in $\mathcal{E}_1$-norm.
			\item[{\rm (iii)}] $\displaystyle \lim_{n\to \infty}\mathbb{E}_x\left[ \sup_{0\leq t\leq T} \left|  {\sf A}^n_t - {\sf A}_t \right|  \right] =0  \text{ for any }T>0\text{ and q.e. }\  x\in E.$
			\item[{\rm (iv)}] $R_1 \mu_n$ converges to $R_1 \mu$ locally uniformly on $E$.
		\end{itemize}
	\end{theorem}
	\begin{proof}
		${\rm (i)} \Rightarrow {\rm (iv)}$: For any $\varepsilon>0$, there exists a compact set $K_1 \subset E$ such that 
		\[
		\sup_{x \in E} R_1(1_{K_1^c} \mu)(x)<\frac{\varepsilon}{5}\quad\text{and}\quad \sup_{n\in\mathbb{N}}\sup_{x\in E }R_1(1_{K_1^c} \mu_n)(x)<\frac{\varepsilon}{5}.
		\]
		Since $(\mathcal{E},\mathcal{F})$ is regular, there exists $\psi \in \mathcal{F} \cap C_0(E)$ such that $0\leq\psi\leq1$ on $E$ and $\psi \equiv 1$ on $K_1$. Then, for any $x\in E$, we have
		\[
		\begin{aligned}
			\left| R_1 \mu_n(x) - R_1 \mu(x)  \right|\leq& \left| \int_{K_1} R_1(x,y)\mu_n({\rm d}y) -\int_{K_1} R_1(x,y)\mu({\rm d}y)   \right|\\
			&+\left| \int_{K_1^c} R_1(x,y)\mu_n({\rm d}y) -\int_{K_1^c} R_1(x,y)\mu({\rm d}y)   \right|\\
			<& \left| \int_{K_1} \psi(y)R_1(x,y)\mu_n({\rm d}y) -\int_{K_1} \psi(y)R_1(x,y)\mu({\rm d}y)   \right|+\frac{2}{5}\varepsilon\\
			\leq&\frac{2}{5}\varepsilon+\left| \int_E \psi(y)R_1(x,y)\mu_n({\rm d}y) -\int_E \psi(y)R_1(x,y)\mu({\rm d}y)   \right|\\
			&\left| \int_{K_1^c} \psi(y)R_1(x,y)\mu_n({\rm d}y) -\int_{K_1^c} \psi(y)R_1(x,y)\mu({\rm d}y)   \right|\\
			<&\frac{4}{5}\varepsilon +\left| \int_E \psi(y)R_1(x,y)\mu_n({\rm d}y) -\int_E \psi(y)R_1(x,y)\mu({\rm d}y)   \right|.
		\end{aligned}
		\]
		Since $\phi(\cdot)R_1(x,\cdot) \in C_0(E)$ for fixed $x \in E$, then there exists $N\in\mathbb{N}$ such that for any $n \geq N$,
		\[
		\left| \int_E \psi(y)R_1(x,y)\mu_n({\rm d}y) -\int_E \psi(y)R_1(x,y)\mu({\rm d}y)   \right|<\frac{1}{5}\varepsilon.
		\]
		Therefore, $R_1 \mu_n(x)$ converges to $R_1 \mu(x)$ for all $x \in E$ as $n \to \infty$. Similarly, for any $\varepsilon>0$, there exists a compact set $K_2 \subset E$ such that 
		\[
		\sup_{x \in E} R_1(1_{K_2^c} \mu)(x)<\frac{\varepsilon}{3}\quad\text{and}\quad \sup_{n\in\mathbb{N}}\sup_{x\in E }R_1(1_{K_2^c} \mu_n)(x)<\frac{\varepsilon}{3}.
		\]
		For any compact set $F \in E$, since $R_1(x,y)$ is continuous on $E \times E$, then $R_1(x,y)$ is uniformly continuous on $F \times K_2$. There exists $\delta>0$ such that for any $x,y \in F$ satisfying ${\sf d}(x,y)<\delta$, we have 
		\[
		\left| R_1(x,z)- R_1(y,z)  \right|<\frac{\varepsilon}{3\sup_{n} \mu_n(K)}\quad\text{for}\quad z \in K.
		\]
		Then, we obtian that 
		\[
		\begin{aligned}
			&\left| R_1 \mu_n(x)-R_1 \mu_n(y)  \right|\\
			\leq&\left|\int_{K_2} R_1(x,z)\mu_n({\rm d} z)-\int_{K_2} R_1(y,z)\mu_n({\rm d} z)   \right|+\left|\int_{K_2^c} R_1(x,z)\mu_n({\rm d} z)-\int_{K_2^c} R_1(y,z)\mu_n({\rm d} z) \right|\\
			<&\frac{2}{3}\varepsilon+\int_{K_2} \left|R_1(x,z)-R_1(y,z)\right|\mu_n({\rm d}z)<\varepsilon.
		\end{aligned}
		\]
		Then, for any compact set $F$, $R_1 \mu_n(x)$ is uniformly equicontinuous on $F$. Therefore, from Arzel\`a–Ascoli Theorem, $R_1 \mu_n$ converges to $R_1 \mu$ uniformly on any compact set $F$. 
		\par ${\rm (iv)} \Rightarrow {\rm (ii)}$: Since $E$ is locally compacet space, there exists monotonically increasing sequence$E_k$ of compact sets such that $\displaystyle E=\bigcup_{k=1}^\infty E_k$. Then, by monotone convergence Theorem, we have
		\[
		\lim_{k\to \infty}\int_{E_k^c} R_1(1_{E_k^c}\mu)\ {\rm d}\mu=0.
		\]
		Hence, for any $\varepsilon>0$, there exists compact sets $K \subset E$ such that 
		\[
		\int_{K^c} R_1(1_{K^c} \mu) \ {\rm d}\mu<\varepsilon\quad\text{and}\quad\sup_{n \in \mathbb{N}} \int_{K^c}  R_1(1_{K^c} \mu_n)\ {\rm d}\mu_n<\varepsilon.
		\]
		We begin by proving the uniform boundedness of the $\mathcal{E}_1$-energy, 
		\[
		\begin{aligned}
			\mathcal{E}_1(U_1 \mu_n,U_1 \mu_n)&=\int_{K} R_1 \mu_n{\rm d}\mu_n + \int_{K^c} R_1 \mu_n\ {\rm d}\mu_n\\
			&= \int_{K} R_1 \mu_n{\rm d}\mu_n+\int_{K^c}\left(R_1 \mu_n^K+R_1 \mu_n^{K^c}  \right) \ {\rm d}\mu_n\\
			&=\int_{K} R_1 \mu_n{\rm d}\mu_n+\int_{K} R_1 \mu_n^{K^c}\ {\rm d}\mu_n+\int_{K^c} R_1 \mu_n^{K^c}\ {\rm d}\mu_n\\
			&\leq2\cdot\sup_{n \in \mathbb{N}}\Vert R_1 \mu_n \Vert_\infty\cdot \sup_{n \in \mathbb{N}}\mu_n(K) +\varepsilon<\infty.
		\end{aligned}
		\]
		Hence, $M_1:=\sup_{n} \mathcal{E}_1(U_1 \mu_n,U_1 \mu_n)<\infty$. Now, we consider that
		\[
		\begin{aligned}
			\mathcal{E}_1(U_1 \mu_n - U_1 \mu,U_1 \mu_n - U_1 \mu)&=\int_E R_1 \mu_n - R_1 \mu \ {\rm d} 
			\mu_n + \int_E R_1 \mu_n - R_1 \mu \ {\rm d} \mu \\
			&= \int_K R_1 \mu_n - R_1 \mu \ {\rm d} \mu_n+\int_K R_1 \mu_n - R_1 \mu \ {\rm d} \mu \\
			&+\int_{K^c} R_1 \mu_n - R_1 \mu \ {\rm d} \mu_n +\int_{K^c} R_1 \mu_n - R_1 \mu \ {\rm d} \mu\\
			&=:{\rm I_1+I_2+I_3+I_4}.
		\end{aligned}
		\]
		First, we can see that
		\[
		{\rm I_1}\leq \Vert R_1 \mu_n - R_1 \mu \Vert_{K,\infty} \cdot \sup_{n \in \mathbb{N}} \mu_n(K)\to 0\quad\text{as}\quad n \to \infty
		\]
		and
		\[
			{\rm I_2}\leq \Vert R_1 \mu_n - R_1 \mu \Vert_{K,\infty} \cdot  \mu(K)\to 0\quad\text{as}\quad n \to \infty.
		\]
		On the other hand,  we obtain that
		\[
		\begin{aligned}
			|{\rm I_3}|&\leq \left|\int_{K^c} R_1 \mu_n\ {\rm d}\mu_n\right| +\left|\int_{K^c} R_1 \mu\ {\rm d}\mu_n\right|\\
			&=|\mathcal{E}_1(U_1 \mu_n,U_1 \mu_n^{K^c})| +|\mathcal{E}_1 (U_1 \mu,U_1 \mu_n^{K^c})|\\
			&\leq\sqrt{\mathcal{E}_1(U_1 \mu_n,U_1 \mu_n)}\cdot\sqrt{\mathcal{E}_1(U_1 \mu_n^{K^c},U_1 \mu_n^{K^c})}+\sqrt{\mathcal{E}_1(U_1 \mu,U_1 \mu)}\cdot \sqrt{\mathcal{E}_1(U_1 \mu_n^{K^c},U_1 \mu_n^{K^c})}\\
			&<\sqrt{\varepsilon M_1}+\sqrt{\varepsilon M_2},
		\end{aligned}
		\]
		where $M_2:=\mathcal{E}_1(U_1 \mu,U_1 \mu)$. Similarly, we can have $|{\rm I_3}| <\sqrt{\varepsilon M_1}+\sqrt{\varepsilon M_2}$. As $\varepsilon \to 0$, both $I_3,I_4\to 0$. Therefore, by first letting $n \to \infty$ and then passing to the limit $\varepsilon \to 0$, we obtain the desired result.
		\par ${\rm (ii)} \Rightarrow {\rm (iii)}$ is from Theorem~\ref{thm3.3} and ${\rm (ii)} \Rightarrow {\rm (i)}$ is from Proposition~\ref{prop 2.1}. Finally, ${\rm (iii) \Rightarrow (iv)}$: For any $T>0$, from integration by parts  formula, we have
		\[
		\int_0^T e^{-t} {\rm d}{\sf A}^n_t=e^{-T}{\sf A}^n_T +\int_0^T e^{-t}{\sf A}_t^n \ {\rm d}t \quad\text{and}\quad \int_0^T e^{-t} {\rm d}{\sf A}_t=e^{-T}{\sf A}_T +\int_0^T e^{-t}{\sf A}_t \ {\rm d}t. 
		\]
		By assumption, for q.e. $x\in E$,
		\[
		\begin{aligned}
			\left| \mathbb{E}_x\left[ \int_0^T e^{-t}\ {\rm d}{\sf A}_t^n\right]  -\mathbb{E}_x \left[ \int_0^T e^{-t}\ {\rm d}{\sf A}_t\right] \right|&\leq \mathbb{E}_x\left[ \left| e^{-T}({\sf A}_T^n -{\sf A}_T) +\int_0^T e^{-t} \left({\sf A}_t^n-{\sf A}_t\right) \ {\rm d}t  \right|  \right]\\
			&\leq 2\cdot\mathbb{E}_x \left[ \sup_{0\leq t\leq T}\left| {\sf A}_t^n -{\sf A}_t  \right|\right]\\
			&\to 0 \quad\text{as}\quad n\to\infty.
		\end{aligned}
		\]
		Hence, we have
		\[
		\lim_{n \to \infty}\mathbb{E}_x\left[ \int_{0}^{T} e^{-t}\ {\rm d}{\sf A}_t^n  \right]=\mathbb{E}_x\left[ \int_{0}^{T} e^{-t}\ {\rm d}{\sf A}_t  \right]\quad \text{for q.e. }x\in E.
		\]
		Consider, for q.e. $x \in E$
		\[
		\mathbb{E}_x\left[ \int_{T}^{\infty} e^{-t}\ {\rm d}{\sf A}_t^n  \right]\leq e^{-T} \cdot\sup_{n \in \mathbb{N}}\Vert R_1 \mu_n \Vert_\infty\quad\text{and}\quad \mathbb{E}_x\left[ \int_{T}^{\infty} e^{-t}\ {\rm d}{\sf A}_t  \right]\leq e^{-T} \Vert R_1 \mu \Vert_\infty.
		\]
		Therefore,   $\mathbb{E}_x\left[ \int_{T}^{\infty} e^{-t}\ {\rm d}{\sf A}_t^n  \right] \to 0$ as $T \to \infty$ for each $n$ and $\mathbb{E}_x\left[ \int_{T}^{\infty} e^{-t}\ {\rm d}{\sf A}_t  \right] \to 0$ as $T \to \infty$ . Then, for any $\varepsilon>0$, there exists $T_0>0$ such that for any $n\in\mathbb{N}$,
		\[
		\mathbb{E}_x\left[ \int_{T_0}^{\infty} e^{-t}\ {\rm d}{\sf A}_t^n  \right] <\frac{\varepsilon}{3}\quad\text{and}\quad\mathbb{E}_x\left[ \int_{T_0}^{\infty} e^{-t}\ {\rm d}{\sf A}_t  \right] <\frac{\varepsilon}{3}.
		\]
		For any $\varepsilon>0$ and q.e. $x \in E$, there exists $N\in\mathbb{N}$ such that for any  $n\geq N$,
		\[
		\left| \mathbb{E}_x\left[ \int_{0}^{T} e^{-t}\ {\rm d}{\sf A}_t  \right] - \mathbb{E}_x\left[ \int_{0}^{T} e^{-t}\ {\rm d}{\sf A}_t^n  \right]\right|<\frac{\varepsilon}{3}.
		\]
		Therefore,
		\[
		\begin{aligned}
			&\left| \mathbb{E}_x\left[ \int_{0}^{\infty} e^{-t}\ {\rm d}{\sf A}_t^n  \right] -\mathbb{E}_x\left[ \int_{0}^{\infty} e^{-t}\ {\rm d}{\sf A}_t  \right]  \right|\\
			\leq&\left| \mathbb{E}_x\left[ \int_{0}^{T_0} e^{-t}\ {\rm d}{\sf A}_t^n  \right] -\mathbb{E}_x\left[ \int_{0}^{T_0} e^{-t}\ {\rm d}{\sf A}_t  \right]  \right|+\mathbb{E}_x\left[ \int_{T_0}^{\infty} e^{-t}\ {\rm d}{\sf A}_t^n\right]+\mathbb{E}_x\left[ \int_{T_0}^{\infty} e^{-t}\ {\rm d}{\sf A}_t\right]\\
			<& \frac{\varepsilon}{3}+\frac{\varepsilon}{3}+\frac{\varepsilon}{3}=\varepsilon.
		\end{aligned}
		\]
		Therefore, $\lim_{n \to \infty} R_1 \mu_n(x)= R_1 \mu(x)$ for q.e. $x \in E$. Moreover,  $\{ R_1 \mu_n\}$ is equicontinuous on every compact set $K$ (please refer to proofs ${\rm (i)} \Rightarrow{\rm (iv)}$ above). For any $\varepsilon>0$, there exists $N\subset E$ with ${\rm Cap}(N)=0$ and $N\in \mathbb{N}$ such that for any $x\in E\setminus N$ and $n\geq N$, we have
		\[
		\left| R_1\mu_n(x) -R_1\mu(x)  \right|<\varepsilon.
		\] 
		Since the complement of $N$ is dense in $E$ and $R_1\mu^n,R_1\mu \in C_b(E)$, then for any $x \in K$, there exists $x_k \in E\setminus N$ such that $|R_1\mu_n(x)-R_1\mu_n(x_k)|<1/k$ and $|R_1\mu(x)-R_1\mu(x_k)|<1/k$. Hence, for any $x\in K$ and $n\geq N$, we have
		\[
		\begin{aligned}
			\left|R_1\mu_n-R_1\mu(x)\right|&\leq\left| R_1\mu_n(x)-R_1\mu_n(x_k)    \right|+\left| R_1\mu_n(x_k)-R_1\mu(x_k)    \right|+\left|  R_1\mu(x_k)- R_1\mu(x) \right|\\
			&\leq\frac{2}{k}+\varepsilon\to \varepsilon\quad\text{as}\quad k \to \infty.
		\end{aligned}
		\]
		Therefore, $R_1\mu_n$ converges to $R_1\mu$ uniformly on $K$. From Arzel\`a–Ascoli Theorem, we can get the conclusion. 
		\end{proof}

		\begin{remark}
			Similar to Theorem~\ref{thm4.2}, we have the following additional observation concerning almost sure convergence in connection with the equivalence established in Theorem~\ref{thm4.3}. Suppose that in addition to the assumptions of Theorem~\ref{thm4.3}, the energy condition (\ref{4.2}) holds. Then, by virtue of the arguments in Remark~\ref{rem3.1}, $\mathcal{E}_1^{1/2}$-norm strongly convergence of $R_1\mu_n$ can be strengthened to hold for q.e. $x\in E$:
			\[
			\mathbb{P}_x\left(\lim_{n\to\infty} \mathsf{A}^n_t = \mathsf{A}_t\text{ locally uniformly in } t \text{ on } [0,+\infty)\right)=1\quad\text{for q.e. } x\in E.
			\]
			Consequently, under this energy summability condition \ref{4.2}, from Theorem~\ref{thm4.1}, the above almost sure convergence of the full sequence (for q.e. $x\in E$) is also equivalent to the four conditions (i)--(iv) of Theorem~\ref{thm4.3}.
			If the summability condition (\ref{4.2}) is not imposed, we can still conclude the existence of a subsequence $\{n_k\}$ such that the almost sure convergence of PCAFs along this subsequence (for q.e. $x\in E$) is equivalent to the corresponding subsequential versions of the four conditions (i)--(iv). In other words, the five conditions, namely (i)--(iv) and the almost sure convergence along the subsequence $\{n_k\}$ are all equivalent on this subsequence.
		\end{remark}

		\section*{Acknowledgments}
		The author would like to express sincere gratitude to Professor Kazuhiro Kuwae for his guidance, and to Professor Kaneharu Tsuchida and Professor Takumu Ooi for their valuable discussions.

	 \bibliographystyle{amsplain}
	\bibliography{ref.bib}
\end{document}